\documentclass[12pt]{article}
\usepackage{latexsym,amssymb,amsthm,enumerate,float,geometry,cite}
\newtheorem{theorem}{Theorem}

\newtheorem{lemma}[theorem]{Lemma}

\usepackage{lineno}
\usepackage{setspace}

\begin{document}
\onehalfspace

\title{Minimum distance-unbalancedness of trees}
\author{Marie Kramer \and Dieter Rautenbach}
\date{}

\maketitle

\begin{center}
{\small 
Institute of Optimization and Operations Research, Ulm University,\\ 
Ulm, Germany, \texttt{$\{$marie.kramer,dieter.rautenbach$\}$@uni-ulm.de}}
\end{center}

\begin{abstract}
For a graph $G$, and two distinct vertices $u$ and $v$ of $G$,
let $n_G(u,v)$ be the number of vertices of $G$ that are closer in $G$ to $u$ than to $v$.
Miklavi\v{c} and \v{S}parl (arXiv:2011.01635v1)
define the distance-unbalancedness of $G$
as the sum of $|n_G(u,v)-n_G(v,u)|$
over all unordered pairs of distinct vertices $u$ and $v$ of $G$.
Confirming one of their conjectures,
we show that the stars minimize the distance-unbalancedness
among all trees of a fixed order.\\[3mm]
\end{abstract}

\section{Introduction}
Inspired by several graph theoretical notions studied in mathematical chemistry,
and, especially, by the notions of distance-balanced graphs \cite{ha,je} 
and the Mostar index of a graph \cite{do},
Miklavi\v{c} and \v{S}parl \cite{misp2} introduced the 
distance-unbalancedness of a graph $G$.
Here, we confirm one of their conjectures from \cite{misp2}.

Before we can explain the distance-unbalancedness as well as our contribution,
we need to introduce some notation.
We consider only finite, simple, and undirected graphs.
For a graph $G$, and two vertices $u$ and $v$ of $G$, 
let ${\rm dist}_G(u,v)$ denote the {\it distance in $G$ between $u$ and $v$}, and 
let $n_G(u,v)$ be the number of vertices $w$ of $G$ 
that are closer to $u$ than to $v$,
that is, that satisfy ${\rm dist}_G(u,w)<{\rm dist}_G(v,w)$.
The {\it Mostar index} \cite{do} of $G$ is
$${\rm Mo}(G)=\sum\limits_{uv\in E(G)}|n_G(u,v)-n_G(v,u)|.$$
A graph $G$ is {\it distance-balanced} \cite{ha,je} 
if $n_G(u,v)=n_G(v,u)$ for every edge $uv$ of $G$;
or, equivalently, if ${\rm Mo}(G)=0$.
The {\it distance-unbalancedness} \cite{misp2} of $G$ is
$${\rm uB}(G)=\sum\limits_{\{ u,v\}\in {V(G)\choose 2}}|n_G(u,v)-n_G(v,u)|,$$
where ${V(G)\choose 2}$ denotes the set of 
all $2$-element subsets of the vertex set $V(G)$ of $G$,
that is, the edge set of the complete graph with vertex set $V(G)$.
A graph $G$ is {\it highly distance-balanced} \cite{misp1} 
if $n_G(u,v)=n_G(v,u)$ for every two distinct vertices $u$ and $v$ of $G$;
or, equivalently, if ${\rm uB}(G)=0$.

For a detailed discussion about the role of the above notions in mathematical chemistry,
we refer to the cited references.
In \cite{misp2} Miklavi\v{c} and \v{S}parl 
collect numerous observations concerning the distance-unbalancedness 
and pose several conjectures. 
Confirming Conjecture 4.2. from \cite{misp2}, we prove the following.

\begin{theorem}\label{theorem1}
If $T$ is a tree of order $n$, then 
$${\rm uB}(T)\geq {\rm uB}(K_{1,n-1})=(n-1)(n-2)$$
with equality if and only if 
$T$ is either a star $K_{1,n-1}$ 
or $n=4$ and $T$ is the path $P_4$.
\end{theorem}
As the definition of distance-unbalancedness involves a summation 
over all unordered pairs of distinct vertices,
this parameter is much harder to approach 
than many other comparable parameters.
In particular, it is much more difficult to analyze the effect 
of the kind of local modifications that are usual proof techniques in this area.
Our proof relies on the insight, 
implicit in Lemma \ref{lemma1} below,
that considering all unordered pairs of vertices
of distance one or two is sufficient.

The rest of this paper is devoted to the proof of Theorem \ref{theorem1}.

\section{Proof of Theorem \ref{theorem1}}

For a graph $G$, the square $G^2$ of $G$ has the same vertex set as $G$,
and two distinct vertices of $G$ are adjacent in $G^2$ if their distance in $G$ is at most two.
For the proof of Theorem \ref{theorem1},
we consider the following auxiliary parameter
$${\rm uB}_2(G)=\sum\limits_{uv\in E(G^2)}|n_G(u,v)-n_G(v,u)|,$$
and we establish the following.
\begin{lemma}\label{lemma1}
If $T$ is a tree of order $n$, then 
${\rm uB}_2(T)\geq (n-1)(n-2)$.
\end{lemma}
Before proving this lemma, we show that Theorem \ref{theorem1} is an immediate consequence.

\begin{proof}[Proof of Theorem \ref{theorem1}]
By definition and Lemma \ref{lemma1},
${\rm uB}(T)\geq {\rm uB}_2(T)\geq (n-1)(n-2)$
for every tree $T$ of order $n$.
It is an easy calculation that stars and $P_4$ 
satisfy ${\rm uB}(T)=(n-1)(n-2)$.
Now, in order to complete the proof, 
we suppose, for a contradiction, that $T$ is a tree of order $n$
with ${\rm uB}(T)=(n-1)(n-2)$ that is neither a star nor $P_4$.
Clearly, this implies that $n\geq 5$, and that $T$ has diameter at least three.
Since ${\rm uB}(T)=(n-1)(n-2)$ implies ${\rm uB}(T)={\rm uB}_2(T)$,
we have $n_T(u,v)=n_T(v,u)$ for every two vertices $u$ and $v$ at distance three.

Let $u$ and $v$ be two vertices at distance three.
If $u$ has a neighbor $u'$ that does not lie on the path $P$ between $u$ and $v$,
and $v'$ is the neighbor of $v$ on $P$,
then $u'$ and $v'$ have distance three but
$n_T(u',v')<n_T(u,v)=n_T(v,u)<n_T(v',u'),$
which is a contradiction.
Using $n_T(u,v)=n_T(v,u)$ 
this easily implies that $T$ arises 
from the disjoint union of two stars of order $\frac{n}{2}$
by adding an edge between the two center vertices.

Now,
\begin{eqnarray*} 
{\rm uB}(T) & \geq & {\rm uB}_2(T)\\
&=& (n-2)^2+2\left(\frac{n}{2}-1\right)^2\\
&=& (n-1)(n-2)+\frac{1}{2}(n-2)(n-4)\\
&>& (n-1)(n-2),
\end{eqnarray*} 
which is a contradiction, and completes the proof.
\end{proof}
We proceed to the proof of the lemma.
\begin{proof}[Proof of Lemma \ref{lemma1}]
Choose the tree $T$ of order $n$ such that ${\rm uB}_2(T)$ is as small as possible.
If $T$ is a path, then a simple calculation yields ${\rm uB}_2(T)=(n-1)(n-2)$, 
and the desired result follows.
Hence, we may assume that $T$ has at least one vertex of degree at least three.

We consider different cases.

\medskip

\noindent {\bf Case 1} {\it $T$ has exactly one vertex $c$ of degree $k$ at least three.}

\medskip

\noindent Let the $k$ components of $T-c$ have orders $n_1,\ldots,n_k$ 
with $n_1\geq \ldots\geq n_k\geq 1$.
Note that all these components are paths, and that $n_1+\cdots+n_k=n-1$.

\medskip

\noindent {\bf Case 1.1} {\it $n_1\leq \frac{n}{2}$.}

\medskip

\noindent We have 
\begin{eqnarray*}
{\rm uB}_2(T) 
& = & \sum\limits_{i=1}^k\Big((n-2)+(n-3)+\cdots+(n-2n_i)\Big)+\sum\limits_{i=1}^{k-1}\sum\limits_{j=i+1}^{k}(n_i-n_j)\\
& \geq & \sum\limits_{i=1}^k\Big((n-2)+(n-3)+\cdots+(n-2n_i)\Big)+(n_1-n_2)\\
& = & \sum\limits_{i=1}^k\Big((2n_i-1)n-n_i(2n_i+1)+1\Big)+(n_1-n_2)\\
& = & f_1(n,k)-\sum\limits_{i=1}^k2n_i^2+(n_1-n_2),
\end{eqnarray*}
where $f_1(n,k)$ is a suitable function of $n$ and $k$.

We consider the following optimization problem:
\begin{eqnarray}\label{e1}
\begin{array}{rrcl}
\min & f_1(n,k)-\sum\limits_{i=1}^k2n_i^2+(n_1-n_2) & & \\[3mm]
s.th. & \frac{n}{2}\geq n_1\geq \ldots\geq n_k&\geq &1 \\
	   & n_1+\cdots+n_k&=&n-1 \\
	   & n_1,\ldots,n_k & \in& \frac{\mathbb{N}}{2}.
\end{array}
\end{eqnarray}
Note that in (\ref{e1}), 
the originally integral values of the $n_i$ 
have been relaxed to being half-integral.

Let $(n_1,\ldots,n_k)$ be a lexicographically maximal optimal solution of (\ref{e1}).

If $n_1<\frac{n}{2}$ and $n_2>n_3$, then
\begin{eqnarray*}
\Bigg(-2\left(n_1+\frac{1}{2}\right)^2-2\left(n_2-\frac{1}{2}\right)^2+\left(n_1+\frac{1}{2}\right)-\left(n_2-\frac{1}{2}\right)\Bigg)
&-&\Bigg(-2n_1^2-2n_2^2+n_1-n_2\Bigg)\\
&=&-2(n_1-n_2)\\
&\leq &0
\end{eqnarray*}
implies that $\left(n_1+\frac{1}{2},n_2-\frac{1}{2},\ldots,n_k\right)$ 
is a lexicographically larger optimal solution of (\ref{e1}),
which is a contradiction.
If $n_1<\frac{n}{2}$, $n_i>1$ for some $i\in \{ 3,\ldots,k\}$, and $i$ is chosen largest with this property, then
\begin{eqnarray*}
\Bigg(-2\left(n_1+\frac{1}{2}\right)^2-2\left(n_i-\frac{1}{2}\right)^2
+\left(n_1+\frac{1}{2}\right)\Bigg)
&-&\Bigg(-2n_1^2-2n_i^2+n_1\Bigg)\\
&=&-2(n_1-n_i)-\frac{1}{2}\\
&<&0
\end{eqnarray*}
implies that $\left(n_1+\frac{1}{2},\ldots,n_i-\frac{1}{2},\ldots,n_k\right)$ 
is a better solution of (\ref{e1}),
which is a contradiction.

Finally, if $n_1=\frac{n}{2}$ and $n_2<\frac{n}{2}-k+1$, 
then $n_i>1$ for some $i\in \{ 3,\ldots,k\}$.
If $i$ is largest with this property, then
\begin{eqnarray*}
\Bigg(-2\left(n_2+\frac{1}{2}\right)^2-2\left(n_i-\frac{1}{2}\right)^2-\left(n_2+\frac{1}{2}\right)\Bigg)
&-&\Bigg(-2n_2^2-2n_i^2-n_2\Bigg)\\
&=&-2(n_2-n_i)-\frac{3}{2}\\
&<&0
\end{eqnarray*}
implies that $\left(n_1,n_2+\frac{1}{2},\ldots,n_i-\frac{1}{2},\ldots,n_k\right)$ 
is a better solution of (\ref{e1}),
which is a contradiction.

These observations imply that 
\begin{enumerate}[(a)]
\item 
either $n\geq 2k$,
$n_1=\frac{n}{2}$,
$n_2=\frac{n}{2}-k+1$, and
$n_3=\ldots=n_k=1$, 
\item 
or $n<2k$,
$n_1=n-k$, and
$n_2=\ldots=n_k=1$.
\end{enumerate}
In the first case,
\begin{eqnarray*}
{\rm uB}_2(T)
& \geq & 
\sum\limits_{i=1}^k\Big((2n_i-1)n-n_i(2n_i+1)+1\Big)+n_1-n_2\\
&\stackrel{(a)}{=}& (n-1)(n-2)+(n-2k)(k-2)\\
&\geq & (n-1)(n-2),
\end{eqnarray*}
and, in the second case,
\begin{eqnarray*}
{\rm uB}_2(T)
& \geq & 
\sum\limits_{i=1}^k\Big((2n_i-1)n-n_i(2n_i+1)+1\Big)+n_1-n_2\\
&\stackrel{(b)}{=}& (n-1)(n-2)+(2k-n)(n-k-1)\\
&\geq & (n-1)(n-2).
\end{eqnarray*}
Altogether, we obtain ${\rm uB}_2(T)\geq (n-1)(n-2)$ as required in both cases.

\medskip

\noindent {\bf Case 1.2} {\it $n_1>\frac{n}{2}$.}

\medskip

\noindent We have 
\begin{eqnarray*}
{\rm uB}_2(T) 
& = & 
\Big((n-2)+\cdots+1+0+1+\ldots+(2n_1-n)\Big)
+\sum\limits_{i=2}^k\Big((n-2)+\cdots+(n-2n_i)\Big)\\
&&+\sum\limits_{i=1}^{k-1}\sum\limits_{j=i+1}^{k}(n_i-n_j)\\
& \geq & 
\Big((n-2)+\cdots+1+0+1+\ldots+(2n_1-n)\Big)
+\sum\limits_{i=2}^k\Big((n-2)+\cdots+(n-2n_i)\Big)\\
&+&\sum\limits_{i=2}^{k}(n_1-n_i)\\
& = & 
\frac{1}{2}(n-1)(n-2)+\frac{1}{2}(2n_1-n)(2n_1-n+1)+(k-1)n_1\\
&&+\sum\limits_{i=2}^k\Big((2n_i-1)n-n_i(2n_i+1)+1-n_i\Big)\\
& = & f_2(n,k)
+2n_1^2-n_1(4n-k)-\sum\limits_{i=2}^k(2n_i^2+2n_i),
\end{eqnarray*}
where 
we used 
$\sum\limits_{i=2}^k 2n_in=(2(n-1)-2n_1)n$,
and $f_2(n,k)$ is a suitable function of $n$ and $k$.

Note that, for $i\in \{ 2,\ldots,k\}$,
we have $n_1+n_i\leq n_1+n_2\leq n-k+1$, and, hence,
$$4(n_1+n_i)-4n+k+2\leq -3k+6<0.$$
If $n_i>1$ for some $i\in\{ 2,\ldots,k\}$, and $i$ is largest with this property, then 
\begin{eqnarray*}
\Bigg(2(n_1+1)^2-(n_1+1)(4n-k)-2(n_i-1)^2-2(n_i-1)\Bigg)
&-&\Bigg(2n_1^2-n_1(4n-k)-2n_i^2-2n_i\Bigg)\\
&=&4(n_1+n_i)-4n+k+2\\
&<&0.
\end{eqnarray*}
This observation implies that 
$$
\begin{array}{rrcl}
\min & f_2(n,k)+2n_1^2-n_1(4n-k)-\sum\limits_{i=2}^k(2n_i^2+2n_i) & & \\[3mm]
s.th. & n_1&>&\frac{n}{2}\\
&n_1\geq \ldots\geq n_k&\geq &1 \\
	   & n_1+\cdots+n_k&=&n-1 \\
	   & n_1,\ldots,n_k & \in& \mathbb{N}
\end{array}
$$
is assumed 
\begin{enumerate}[(c)]
\item 
for $n_1=n-k$ and $n_2=\ldots=n_k=1$.
\end{enumerate}
This implies
\begin{eqnarray*}
{\rm uB}_2(T) 
& \geq & 
\frac{1}{2}(n-1)(n-2)+\frac{1}{2}(2n_1-n)(2n_1-n+1)+(k-1)n_1\\
&&+\sum\limits_{i=2}^k\Big((2n_i-1)n-n_i(2n_i+1)+1-n_i\Big)\\
& \stackrel{(c)}{\geq} & (n-1)(n-2)+(k-1)(k-2)\\ 
& \geq & (n-1)(n-2),
\end{eqnarray*}
and, hence, 
also ${\rm uB}_2(T)\geq (n-1)(n-2)$ as required in this case.

\medskip

\noindent {\bf Case 2} {\it $T$ has at least two vertices of degree at least three.}

\medskip

\noindent Considering two vertices of degree at least three at maximum distance,
it follows that $T$ has a vertex $c$ of degree $k+1$ at least three
such that $T-c$ has 
\begin{itemize}
\item $k$ components that are paths of orders 
$n_1,\ldots,n_k$ with
$n_1\geq \ldots \geq n_k\geq 1$
and 
$$n':=1+n_1+\ldots+n_k\leq \frac{n}{2},$$
as well as 
\item one component $K$ of order $n-n'$.
\end{itemize}
Let $d$ be the neighbor of $c$ in $V(K)$.
Let the tree $T'$ arise from the disjoint union of $K$ 
and a path $P$ of order $n'$ by adding one edge between $d$ 
and an endvertex of $P$.
Our goal is to show that ${\rm uB}_2(T)>{\rm uB}_2(T')$,
which would contradict the choice of $T$, and complete the proof.

We have 
\begin{eqnarray*}
{\rm uB}_2(T)-{\rm uB}_2(T') 
& = & 
\sum\limits_{i=1}^k\Big((n-2)+\cdots+(n-2n_i)+(n-n')-n_i\Big)
+\sum\limits_{i=1}^{k-1}\sum\limits_{j=i+1}^{k}(n_i-n_j)\\
&& -\Big((n-2)+\cdots+(n-(2n'-1))\Big)\\
& \geq & 
\sum\limits_{i=1}^k\Big((n-2)+\cdots+(n-2n_i)+(n-n')-n_i\Big)\\
&& -\Big((n-2)+\cdots+(n-(2n'-1))\Big)\\
&=& 
\sum\limits_{i=1}^k\Big((2n_i-1)n-n_i(2n_i+1)+1+(n-n')-n_i\Big)\\
&&-\Big((2n'-2)n-n'(2n'-1)+1\Big)\\
&=& f_3(n,n',k)-\sum\limits_{i=1}^k2n_i^2,
\end{eqnarray*}
where $f_3(n,n',k)$ is a suitable function of $n$, $n'$, and $k$.

By the convexity of $x\mapsto x^2$, 
$$
\begin{array}{rrcl}
\min & f_3(n,n',k)-\sum\limits_{i=1}^k2n_i^2 & & \\[3mm]
s.th. & n_1\geq \ldots\geq n_k&\geq &1 \\
	   & n_1+\cdots+n_k&=&n'-1 \\
	   & n_1,\ldots,n_k & \in& \mathbb{N}
\end{array}
$$
is assumed 
\begin{enumerate}[(d)]
\item 
for $n_1=n'-k$ and $n_2=\ldots=n_k=1$.
\end{enumerate}
Note that $3n'=2n'+n'\geq 2(k+1)+3\geq 2k+5$.

Now, we obtain
\begin{eqnarray*}
{\rm uB}_2(T)-{\rm uB}_2(T') 
& \geq &
\sum\limits_{i=1}^k\Big((2n_i-1)n-n_i(2n_i+1)+1+(n-n')-n_i\Big)\\
&&-\Big((2n'-2)n-n'(2n'-1)+1\Big)\\
& \stackrel{(d)}{\geq} & (3n'-2k-3)(k-1)\\
& \stackrel{k\geq 2}{>} & 0,
\end{eqnarray*}
which is the desired contradiction,
completing the proof.
\end{proof}

\end{document}